\documentclass[a4paper,UKenglish,cleveref, autoref]{lipics-v2019}
\usepackage[utf8]{inputenc}
\usepackage{amsmath}
\usepackage{amsthm,comment}
\usepackage{tikz}
\usepackage{slashbox}
\usepackage{boxedminipage}
\newcommand{\dist}{{\mbox{dist}}}
\newcommand{\NP}{{\sf NP}}
\newcommand{\PSPACE}{{\sf PSPACE}}
\newcommand{\NEXPTIME}{{\sf NEXPTIME}}
\newcommand{\cP}{{\sf P}}
\newcommand{\ssi}{\subseteq_i}

\newtheorem{oproblem}{Open Problem}

\title{Colouring Generalized Claw-Free Graphs and Graphs of Large Girth: Bounding the Diameter \footnote{Some of the results in this paper appeared in an extended abstract in the proceedings of MFCS 2019~\cite{DPS19}.}} 

\titlerunning{Colouring Generalized Claw-Free Graphs and Graphs of Large Girth}

\author{Barnaby Martin}{Department of Computer Science, Durham University, United Kingdom}{barnaby.d.martin@durham.ac.uk}{}{}
\author{Dani\"el Paulusma}{Department of Computer Science, Durham University, United Kingdom}{daniel.paulusma@durham.ac.uk}{0000-0001-5945-9287}{supported by the Leverhulme Trust (RPG-2016-258).}
\author{Siani Smith}{Department of Computer Science, Durham University, United Kingdom}{siani.smith@durham.ac.uk}{}{}

\authorrunning{B. Martin, D. Paulusma and S. Smith }

\Copyright{Barnaby Martin, Daniel Paulusma and Siani Smith}



\nolinenumbers 
\hideLIPIcs

\begin{document}

\maketitle

\begin{abstract}
For a fixed integer, the {\sc $k$-Colouring} problem is to decide if the vertices of a graph can be coloured with at most $k$ colours for an integer~$k$, such that no two adjacent vertices are coloured alike. A graph~$G$ is $H$-free if $G$ does not contain $H$ as an induced subgraph.  It is known that for all $k\geq 3$, the {\sc $k$-Colouring} problem is \NP-complete for $H$-free graphs if $H$ contains an induced claw or cycle. The case where $H$ contains a cycle follows from the known result that the problem is \NP-complete even for graphs of arbitrarily large fixed girth. We examine to what extent the situation may change if in addition the input graph has bounded diameter. 
\end{abstract}


\section{Introduction}\label{s-intro}

Graph colouring is one of the best studied concepts in Computer Science and Mathematics. This is mainly due to its many practical and theoretical applications and its many natural variants and generalizations. Over the years, numerous surveys and books on graph colouring were published (see, for example,~\cite{Al93,C14,JT95,KTV99,Pa15,RS04b,Tu97}). 

A {\em (vertex) colouring} of a graph $G=(V,E)$ is a mapping $c: V\rightarrow\{1,2,\ldots \}$ that assigns each vertex~$u\in V$ a {\it colour} $c(u)$ in such a way that $c(u)\neq c(v)$ whenever $uv\in E$. If $1\leq c(u)\leq k$, then $c$ is said to be a {\it $k$-colouring} of $G$ and $G$ is said to be $k$-{\it colourable}. The {\sc Colouring} problem is to decide if a given graph $G$ has a $k$-colouring for some given integer~$k$. If $k$ is {\it fixed}, that is, $k$ is not part of the input, we denote the problem by $k$-{\sc Colouring}.
It is well known that even {\sc $3$-Colouring} is \NP-complete~\cite{Lo73}.

In this paper we aim to increase our understanding of the computational hardness of {\sc Colouring}. One way to do this is to consider inputs from families of graphs to learn more about the kind of graph structure that causes the hardness. This led to a highly extensive study of {\sc Colouring} and {\sc $k$-Colouring} for many special graph classes. The best-known result in this direction is due to  Gr\"otschel, Lov\'asz, and Schrijver, who proved that {\sc Colouring} is polynomial-time solvable for perfect graphs~\cite{GLS84}.

Perfect graphs form an example of a graph class that is closed under vertex deletion. Such graph classes are also called {\it hereditary}.
Hereditary graph classes are ideally suited for a {\it systematic} study in the computational complexity of graph problems. 
Not only do they capture a very large collection of  many well-studied graph~classes, but they are also exactly the graph classes that can be characterized by a unique set ${\cal H}$ of minimal forbidden induced subgraphs. When solving an \NP-hard problem under input restrictions, it is standard practice to consider, for example, first the case where ${\cal H}$ has small size, or where each $H\in {\cal H}$ has small size. 

We note that the set ${\cal H}$ defined above may be infinite. For example, the class of bipartite graphs is hereditary and for this class the set ${\cal H}$ consists of all odd cycles.
If ${\cal H}=\{H_1,\ldots,H_p\}$ for some positive integer $p$,  then the corresponding hereditary graph class ${\cal G}$ is said to be {\it finitely defined}.
Formally, a graph $G$ is {\it $(H_1,\ldots,H_p)$-free} if for each $i\in \{1,\ldots,p\}$, $G$ is {\it $H_i$-free}, where the latter means that $G$ does not contain an induced subgraph isomorphic to~$H_i$.

We emphasize that the borderline between \NP-hardness and tractability is often far from clear beforehand and jumps in computational complexity can be extreme. In order to illustrate this behaviour of graph problems, we presented in~\cite{MPS} examples of \NP-hard, \PSPACE-complete and 
\NEXPTIME-complete problems that become even constant-time solvable for every hereditary graph class that is not equal to the class of all graphs.

In this paper, we consider the problems {\sc Colouring} and {\sc $k$-Colouring}. In order to describe known results and our new results we first give some terminology and notation.

\subsection{Terminology and Notation}\label{s-term}
The {\it disjoint union} of two vertex-disjoint graphs $F$ and $G$ is the graph $G+F=(V(F)\cup V(G),E(F)\cup E(G))$.
The disjoint union of~$s$ copies of a graph~$G$ is denoted~$sG$. A {\it linear forest} is the disjoint union of paths.
The {\it length} of a path or a cycle is the number of its edges. The {\it distance} $\dist(u,v)$ between two vertices $u,v$ in a graph~$G$ is the length of a shortest induced path between them. The {\it diameter} of a graph~$G$ is the maximum distance over all pairs of vertices in $G$. The diameter of a disconnected graph is $\infty$.
 The {\it girth} of a graph~$G$ is the length of a shortest induced cycle of~$G$. The girth of a forest is $\infty$.
The graphs~$C_r$, $P_r$ and~$K_r$ denote the cycle, path and complete graph on~$r$ vertices, respectively.

 \begin{figure}
\begin{tikzpicture}[scale=1]
\draw (0,1)--(-1,1)--(-2,0)--(-1,-1)--(2,-1) (-2,0)--(1,0); \draw[fill=black] (-1,1) circle [radius=2pt] (0,1) circle [radius=2pt] (-2,0) circle [radius=2pt] (-1,0) circle [radius=2pt] (0,0) circle [radius=2pt] (1,0) circle [radius=2pt] (-1,-1) circle [radius=2pt] (0,-1) circle [radius=2pt] (1,-1) circle [radius=2pt] (2,-1) circle [radius=2pt];
\end{tikzpicture}
\caption{The polyad $S_{2,3,4}$.}\label{f-st}
\vspace*{-0.4cm}
\end{figure}
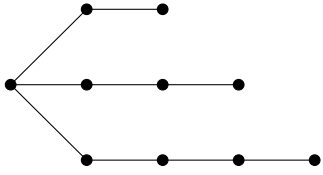

A {\it polyad} is a tree where exactly one vertex has degree at least~$3$. We will use the following special polyads in our paper.
For $r\geq 1$, the graph $K_{1,r}$ denotes the $(r+1)$-vertex {\it star}, that is, the graph with vertices $x,y_1,\ldots,y_r$ and edges $xy_i$ for $i=1,\ldots,r$; here $x$ is called the {\it centre} vertex. The graph $K_{1,3}$ is also called the {\it claw}. The {\it subdivision} of an edge $uw$ in a graph removes $uw$ and replaces it with a new vertex $v$ and edges $uv$, $vw$. 
For $\ell\geq 1$, the graph $K_{1,r}^\ell$ denotes the {\it $\ell$-subdivided star}, which is 
the  graph obtained from a star $K_{1,r}$ by subdividing {\it one} edge of $K_{1,r}$ exactly $\ell$ times.
The graph~$S_{h,i,j}$, for $1\leq h\leq i\leq j$, denotes the {\it subdivided claw}, which is the tree with one vertex~$x$ of degree~$3$ and exactly three leaves, which are of distance~$h$,~$i$ and~$j$ from~$x$, respectively; see Figure~\ref{f-st} for an example.
Note that $S_{1,1,1}=K_{1,3}$. The graph $S_{1,1,2}=K_{1,3}^1$ is also known as the {\it chair}.

A graph $G$ is {\it locally claw-free} if the neighbourhood of every vertex of $G$ induces a claw-free graph.
A graph~$G$ is {\it quasi-claw-free} if every two vertices $u$ and $v$ that are of distance~$2$ from each other have a common neighbour $w$
such that every neighbour of $w$ not in $\{x,y\}$ is adjacent to at least one of $x,y$. 

For a graph $G=(V,E)$, a set $A\subseteq V$ dominates a set $B\subseteq V$ if every vertex of $B$ is either in $A$ or adjacent to a vertex of $A$.
A graph $G=(V,E)$ is {\em almost claw-free}  if the following two
conditions hold:\\[-8pt]
\begin{enumerate}
\item the set consisting of all the centres of induced claws in $G$ is an independent set, and
\item for every $u\in V$, $N(u)$ contains a set of size at most $2$ that dominates $N(u)$.
\end{enumerate}

\noindent
Note that claw-free graphs are locally claw-free, quasi-claw-free and almost claw-free. Hence, the latter three graph classes all generalize the class of claw-free graphs. 
Note also that if all the centres of the induced claws in a graph $G$ form an independent set, then $G$ is locally claw-free.
Hence, as observed by Ryj{\'{a}}cek~\cite{Ry94}, every almost claw-free graph is locally claw-free. Ainouche~\cite{Ai98} showed that both the class of almost claw-free graphs and the class of locally claw-free graphs are incomparable to the class of quasi-claw-free graphs; see also Figure~\ref{f-ai}.

We generalize almost claw-free graphs as follows. 
For an integer~$r\geq 3$, a graph $G=(V,E)$ is {\em almost $K_{1,r}$-free}  if the following two
conditions hold:\\[-8pt]
\begin{enumerate}
\item the set consisting of all the centres of induced $K_{1,r}$s in $G$ is an independent set, and
\item for every $u\in V$, $N(u)$ contains a set of size at most $r-1$ that dominates $N(u)$ .
\end{enumerate}

\begin{center}
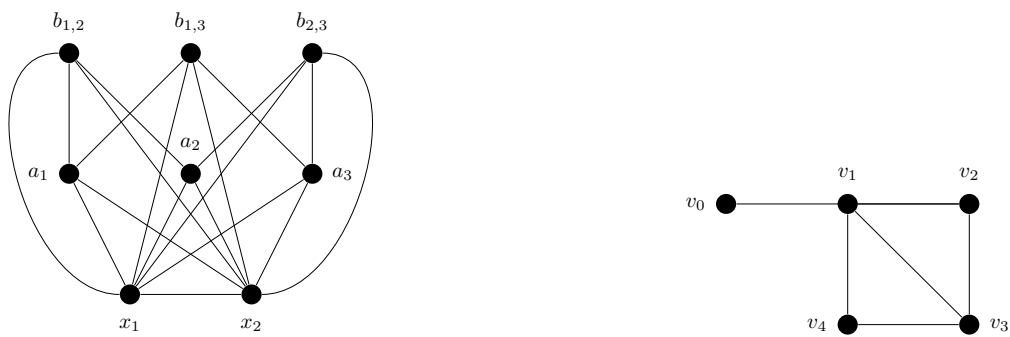
\begin{figure}
\scalebox{.8}{
\begin{tikzpicture}[main_node/.style={circle,draw,minimum size=0.75cm, inner sep=2pt]}]
\node[circle, fill=black](v1) at (0,0){};
\node(v1l) at (-0.5,0){$a_1$};
\node(v2l) at (2,0.5){$a_2$};
\node(v3l) at (4.5,0){$a_3$};
\node(v4l) at (1,-2.5){$x_1$};
\node(v5l) at (3,-2.5){$x_2$};
\node(v6l) at (0,2.5){$b_{1,2}$};
\node(v7l) at (2,2.5){$b_{1,3}$};
\node(v8l) at (4,2.5){$b_{2,3}$};
\node[circle, fill=black](v2) at (2,0){};
\node[circle, fill=black](v3) at (4,0){};
\node[circle, fill=black](v4) at (1,-2){};
\node[circle, fill=black](v5) at (3,-2){};
\node[circle, fill=black](v6) at (0,2){};
\node[circle, fill=black](v7) at (2,2){};
\node[circle, fill=black](v8) at (4,2){};
\draw(v4)--(v2);
\draw(v5)--(v4);
\draw(v4)--(v1);
\draw(v4)--(v3);
\draw(v5)--(v1);
\draw(v5)--(v3);
\draw(v1)--(v6);
\draw(v1)--(v7);
\draw(v2)--(v6);
\draw(v2)--(v8);
\draw(v3)--(v7);
\draw(v3)--(v8);
\draw(v5)--(v2);
\draw(v4) to [out=180, in=180] (v6);
\draw(v4)--(v7);
\draw(v4)--(v8);
\draw(v5) to [out=0, in=0](v8);
\draw(v5)--(v7);
\draw(v5)--(v6);
\end{tikzpicture}
\hspace*{4cm}
\begin{tikzpicture}[main_node/.style={circle,draw,minimum size=0.55cm, inner sep=2pt]}]
\node(v0l) at (-2.5,0){$v_0$};
\node(v4l) at (-0.5,-2){$v_4$};
\node(v1l) at (0,0.5){$v_1$};
\node(v2l) at (2,0.5){$v_2$};
\node(v3l) at (2.5,-2){$v_3$};
\node[circle, fill=black](v0) at (0,0){};
\node[circle, fill=black](v1) at (2,0){};
\node[circle, fill=black](v2) at (0,-2){};
\node[circle, fill=black](v3) at (2,-2){};
\node[circle, fill=black](v4) at (-2,0){};
\draw(v0)--(v1);
\draw(v0)--(v2);
\draw(v0)--(v3);
\draw(v1)--(v4);
\draw(v1)--(v3);
\draw(v2)--(v3);
\end{tikzpicture}}
\caption{{\it Left:} a quasi-claw-free graph that is not locally claw-free and hence not almost claw-free. 
The fact that this graph is quasi-claw-free follows from analysing the pairs of vertices of distance~$2$ from each other: 
for  any two $a$-type vertices, take their common $b$-type neighbour to satisfy the definition; for an $a$-type vertex and its unique non-neighbour of $b$-type: take $x_1$ (or $x_2$); and for two $b$-type vertices: take their common $a$-type neighbour.
The fact that the graph is not locally claw-free can be seen by considering, for example, the neighbourhood of $x_1$.
{\it Right:} an almost claw-free and thus locally claw-free graph that is not quasi-claw-free. The fact that the graph is almost claw-free can be readily checked. The fact that the graph is not quasi-claw-free can be seen by considering the vertices $v_0$ and $v_4$, which only have $v_1$ as a  common neighbour, while neither $v_0$ nor $v_4$ is adjacent to neighbour $v_2$ of $v_1$.
The second example was given by Ainouche~\cite{Ai98}.}\label{f-ai}
\end{figure}
\end{center}

\subsection{Known Results}\label{s-known}
The computational complexity of {\sc Colouring} has been fully classified for $H$-free graphs: if $H$ is an induced subgraph of $P_1+P_3$ or of $P_4$, then {\sc Colouring} for $H$-free graphs is polynomial-time solvable, and otherwise it is \NP-complete~\cite{KKTW01}.
In contrast, the complexity classification for $k$-{\sc Colouring} restricted to $H$-free graphs is still incomplete.
It is known that for every $k\geq 3$, {\sc $k$-Colouring} for $H$-free graphs is \NP-complete if $H$ contains a cycle~\cite{EHK98} or an induced claw~\cite{Ho81,LG83}. However, the remaining case where $H$ is a linear forest has not been settled yet even if $H$ consists of a single path. For $P_t$-free graphs, the cases $k\leq 2$, $t\geq 1$ (trivial), $k\geq 3$, $t\leq 5$~\cite{HKLSS10}, $k=3$, $6\leq t\leq 7$~\cite{BCMSSZ18} and $k=4$, $t=6$~\cite{CSZ19} are polynomial-time solvable and the cases  $k=4$, $t\geq 7$~\cite{Hu16} and  $k\geq 5$, $t\geq 6$~\cite{Hu16} are \NP-complete. The cases where $k=3$ and $t\geq 8$ are still open. For further details, including for linear forests $H$ of more than one connected component, see the survey paper~\cite{GJPS17}  or some recent papers~\cite{CHSZ18,GOPSSS18,HLS,KMMNPS18}.

Recently, Pilipczuk, Pilipczuk and Rz{\k{a}}\.{z}ewski~\cite{PPR21} gave for every $t\geq 3$,
a quasi-polynomial-time algorithm for {\sc $3$-Colouring} on $P_t$-free graphs. 
Rojas and Stein~\cite{RS} proved that for every odd integer $t\geq 9$, {\sc $3$-Colouring} is polynomial-time solvable for 
$({\cal C}^{odd}_{<t-3},P_t)$-free graphs, where ${\cal C}^{odd}_{<t}$ is the set of all odd cycles on less than $t$ vertices.
This complements a result from~\cite{GPS14} which implies that for every $t\geq 1$, {\sc $3$-Colouring}, or its generalization {\sc List $3$-Colouring}\footnote{See Section~\ref{s-pre} for a definition of the {\sc List $k$-Colouring} problem.}, is polynomial-time solvable for $(C_4,P_t)$-free graphs (see also~\cite{LM17}).

Emden-Weinert, Hougardy and Kreuter~\cite{EHK98} proved
that for all integers $k\geq 3$ and $g\geq 3$, {\sc $k$-Colouring}  is \NP-complete for graphs with girth at least~$g$ and with maximum degree at most~$6k^{13}$ (for more results on {\sc Colouring} for graphs of maximum degree, see~\cite{CC06,DDJP,MR14}). 

\subsection{Our Focus}\label{s-focus}
Our starting point is to look at $H$-free graphs where $H$ contains an induced claw or cycle. In this case, {\sc $k$-Colouring} restricted to $H$-free graphs is \NP-compete for every $k\geq 3$, as mentioned above. However, we re-examine the situation after adding a diameter constraint to our input graphs. If the diameter is~$1$, then $G$ is a complete graph, and {\sc Colouring} becomes trivial. As such, our underlying research question is:

\medskip
\noindent
{\it To what extent does bounding the diameter help making {\sc Colouring} and {\sc $k$-Colouring} tractable?}

\medskip
\noindent
We remark that the subclass of a hereditary graph class that consists of all graphs of diameter at most~$d$ for some constant $d$ may not be hereditary. In order to see this, consider for example the (hereditary) class~${\cal G}$ of graphs of maximum degree at most~$2$ and take its subclass~${\cal G}'$ of graphs of diameter at most~$2$. Then $P_3\in {\cal G}'$ but $2P_1\notin {\cal G}'$.
This fact requires some care in the proof of our results. 

We also note that by a straightforward reduction from {\sc 3-Colouring} one can show that $k$-{\sc Colouring} is \NP-complete for graphs of diameter $d$ for all pairs $(k,d)$ with $k\geq 3$ and $d\geq 2$ except for two cases, namely $(k,d)\in \{(3,2),(3,3)\}$. 
Mertzios and Spirakis~\cite{MS16} settled the case $(k,d)=(3,3)$ by proving that {\sc $3$-Colouring} is \NP-complete even for $C_3$-free graphs of diameter~3. The case $(k,d)=(3,2)$ is still open
(see also~\cite{BKM12,BFGP13,DPR21,MPS21,MS16,Pa15}).

In~\cite{MPS21}, we gave polynomial-time algorithms for the more general problem {\sc List $3$-Colouring} for classes of diameter-$2$ graphs that  in addition are $C_s$-free $(s\in \{5,6\})$ or $(C_4,C_t)$-free $(t\in 3,7,8,9\}$. In the same paper we also proved that for every integer $t\geq 8$, the {\sc $3$-Colouring} problem is \NP-complete on the class of $(C_4,C_6,C_7,\ldots,C_t)$-free graphs of diameter~$4$. 
We refer to~\cite{BGMPS21b,BGMPS21} for results on graph problems closely related to {\sc $3$-Colouring} restricted to graph classes of bounded diameter. These problems include {\sc Near-Bipartiteness}, {\sc Independent Feedback Vertex Set}, {\sc Independent Odd Cycle Transversal}, {\sc Acyclic $3$-Colouring} and {\sc Star $3$-Colouring}.

\subsection{Our Results}\label{s-ours}

We complement the bounded diameter results of Mertzios and Spirakis~\cite{MS16} and Martin et al.~\cite{MPS21} by presenting a set of new colouring results for generalized claw-free graphs and graphs of large girth whose diameter is bounded by a constant.

First, in Section~\ref{s-cycle}, we consider graphs of bounded diameter and girth.
We provide new polynomial-time and \NP-hardness results for {\sc Colouring} and {\sc List Colouring}, identifying and narrowing the gap between tractability and intractability, in particular we consider $3$-{\sc Colouring}  (see also Table~\ref{t-table}).

Second, in Section~\ref{s-polyad}, we research the effect on bounding the diameter of {\sc $k$-Colouring} and {\sc Colouring} restricted to graph classes that generalize the class of claw-free graphs. Our polynomial-time results for {\sc $k$-Colouring} hold in fact for {\sc List $k$-Colouring}. In particular, we prove that for all integers $d,k,r\geq 1$, {\sc List $k$-Colouring} is constant-time solvable for almost $K_{1,r}$-free graphs of diameter at most~$d$.
This result forms the starting point of our investigation in this section. We will show that it cannot be generalized to {\sc Colouring} (when $k$ is part of the input). As such we fix the number of colours~$k$ and consider quasi-claw-free graphs, almost $K_{1,r}$-free graphs, locally claw-free graphs and polyad-free graphs for various larger polyads. This leads to a number of new polynomial-time and \NP-complete results for $k$-{\sc Colouring}. For our results on polyads, we also refer to Table~\ref{fig:triad-summary}. 

By working in a systematic way, our results in Sections~\ref{s-cycle} and~\ref{s-polyad} exposed a number of natural open problems. In Section~\ref{s-con} we discuss directions for future work and summarize these questions.

\begin{center}
\begin{table}
\begin{tabular}{ |c|c|c|c|c|c|c|c|c|c|c|} 
 \hline
\backslashbox{diameter}{girth} & $\geq 3$ & $\geq 4$ & $\geq 5$ & $\geq 6$ & $\geq 7$ & $\geq 8$ & $\geq 9$ &$\geq 10$ &$\geq 11$ &$\geq 12$   \\ 
 \hline
 $\leq 1$ & \cP & \cP & \cP & \cP & \cP & \cP & \cP & \cP& \cP& \cP   \\
 \hline
 $\leq 2$ &  ? & ? & \cP &\cP & \cP & \cP & \cP & \cP& \cP& \cP   \\
 \hline
 $\leq 3$ & \NP-c & \NP-c& ? & ? & \cP & \cP& \cP& \cP& \cP& \cP \\
 \hline
 $\leq 4$ & \NP-c & \NP-c & \NP-c & \NP-c & ? & ? & \cP& \cP& \cP& \cP \\
 \hline
 $\leq 5$ & \NP-c & \NP-c & \NP-c & \NP-c & ? & ? & ?& ?& 
 \cP& 
 \cP \\
\hline
\end{tabular}
\vspace*{3mm}
\caption{The  complexity of $3$-{\sc{Colouring}} for graphs of diameter at most~$d$ and girth at least~$g$
for fixed values of $d$ and $g$. Here, \cP, \NP-c and ? represent polynomial, \NP-complete and open cases, respectively. 
The results in this table are generalized to larger values of $d$ (see Theorem~\ref{t-main2}).
}\label{t-table}
\end{table}
\end{center}

\section{Preliminaries}\label{s-pre}

In this section we complement Section~\ref{s-term} by giving some additional terminology and notation.
We also recall some useful results from the literature.

Let $G=(V,E)$ be a graph. A vertex $u\in V$ is {\it dominating} if $u$ is adjacent to every other vertex of $G$. 
For a set $S\subseteq V$, the graph $G[S]=(S,\{uv\; |\; u,v\in S\; \mbox{and}\; uv\in E\})$ denotes the subgraph of $G$ induced by $S$.  
The {\it neighbourhood} of a vertex $u\in V$ is the set $N(u)=\{v\; |\; uv\in E\}$ and the {\it degree} of $u$ is the size of $N(u)$.
For a set $U\subseteq V$, we write $N(U)=\bigcup_{u\in U}N(u)\setminus U$.
For a set $U\subseteq V$ and a vertex $u\in U$, the {\it private neighbourhood} of $u$ with respect to $U$ is the set
$N(u)\setminus (N(U\setminus \{u\})\cup U)$ of {\it private neighbours} of $u$ with respect to $U$, 
which is the set of neighbours of $u$ outside $U$ that are not a neighbour of any other vertex of $U$. 
If every vertex of $G$ has degree~$p$, then $G$ is {\it ($p$)-regular}.

\begin{table}
\begin{center}
\begin{tabular}{|c|c|c|c|c|} 
 \hline
Colours & Diameter & $H$-free & Complexity & Theorem \\ 
\hline
fixed $k$ & $d$ & $K_{1,r}$ & P & \ref{t-constant} \\
\hline
input $k$ & $d$ & $K_{1,4}$ & NP-c & \ref{t-col} \\
\hline
$3$ & $d$ & $K_{1,3}^1$ & P & \hspace*{4mm}\ref{s112}(1) \\
\hline 
$3$ & $2$ & $K_{1,r}^2$ & P &  \hspace*{4mm}\ref{s112}(2)\\
\hline 
$3$ & $4$ & $K_{1,4}^3$ & \NP-c &  \hspace*{4mm}\ref{s112}(3) \\
\hline
$4$ & $2$ & $K_{1,3}^1$ & \NP-c &  \hspace*{4mm}\ref{s112}(4) \\
\hline 
$3$ & $2$ & $S_{1,2,2}$ & P & \ref{s122} \\
\hline
\end{tabular}
\end{center}
\vspace*{3mm}
\caption{Our polynomial-time (\cP) and \NP-complete (\NP-c) results for polyad-free graphs.}
\label{fig:triad-summary}
\end{table}

The {\it diamond} is the graph obtained from the $K_4$ after removing an edge. The {\it bull} is the graph obtained from a triangle on vertices $x,y,z$ after adding two new vertices $u$ and $v$ and edges $xu$ and $yv$.

A {\it clique} in a graph is a set of pairwise adjacent vertices, and an {\it independent set} is a set of pairwise non-adjacent vertices.
By Ramsey's Theorem~\cite{Ra30},  there exists a constant, which we denote by $R(k,r)$, such that any graph on at least $R(k,r)$ vertices contains either a clique of size $k$ or an independent set of size $r$. 
A cycle is {\it odd} if it has odd length.

We will use the aforementioned results of Kr\'al' et al.; Holyer; Leven and Galil; Emden-Weinert, Hougardy and Kreuter; and Mertzios and Spirakis. If a graph $H$ is an induced subgraph of a graph $H'$, then we use $H\ssi H'$ to denote this.
 
\begin{theorem}[\cite{KKTW01}]\label{t-dichotomy}
Let $H$ be a graph. If $H\ssi  P_4$ or $H\ssi P_1+ P_3$, then
{\sc Colouring} restricted to $H$-free graphs is polynomial-time solvable, otherwise it is \NP-complete. 
\end{theorem}

\begin{theorem}[\cite{Ho81,LG83}]\label{t-claw}
For every integer $k\geq 3$, $k$-{\sc Colouring} is \NP-complete for claw-free graphs.
\end{theorem}

\begin{theorem}[\cite{EHK98}]\label{t-girth1}
For all integers $k\geq 3$ and $g\geq 3$, {\sc $k$-Colouring}  is \NP-complete for graphs with girth at least~$g$ (and with maximum degree at most~$6k^{13}$).
\end{theorem}

\begin{theorem}[\cite{MS16}]\label{t-ms16}
{\sc $3$-Colouring} is \NP-complete for $C_3$-free graphs of diameter~$3$.
\end{theorem}

A {\it list assignment} of a graph $G=(V,E)$ is a function $L$ that prescribes a {\it list of admissible colours} $L(u)\subseteq \{1,2,\ldots\}$ to each $u\in V$.
A colouring $c$  {\it respects} ${L}$ if  $c(u)\in L(u)$ for every $u\in V$. For an integer $\ell\geq 1$, we say that $L$ is an {\it $\ell$-list assignment} 
if $|L(u)|\leq \ell$ for each $u\in V$. For an integer $k\geq 1$, we say that $L$ is an {\it list $k$-assignment} 
if $L(u)\subseteq \{1,\ldots,k\}$ for each $u\in V$.
The {\sc List Colouring} problem is to decide if a graph $G$ with a list assignment~$L$ has a colouring that respects $L$.
For a fixed integer $\ell\geq 1$, the {\sc $\ell$-List Colouring} problem is to decide if a graph $G$ with an $\ell$-list assignment $L$ has a colouring that respects $L$. 
For a fixed integer $k\geq 1$, the {\sc List $k$-Colouring} problem is to decide if a graph $G$ with a list $k$-assignment $L$ has a colouring that respects $L$. Note that {\sc $k$-Colouring} is a special case of {\sc List $k$-Colouring} (and that the latter is a special case of {\sc $k$-List Colouring}).

Our strategy for obtaining polynomial-time algorithms for {\sc List  $3$-Colouring} for special graph classes is often to reduce the input to a polynomial number of instances of $2$-{\sc List Colouring}. The reason is that we can then apply the following 
 well-known result of Edwards.

\begin{theorem}[\cite{Ed86}]\label{t-2sat}
The {\sc $2$-List Colouring} problem is linear-time solvable.
\end{theorem}

We will also use the following result, which includes the Hoffman-Singleton Theorem, which provides a description of regular graphs of diameter~$2$ and girth~$5$.
\begin{theorem}[\cite{D73,HS60,S68}]\label{t-all}
For every $d\geq 1$, every graph of diameter $d$ and girth $2d+1$ is $p$-regular for some integer~$p$.
Moreover, if $d=2$, then there are only four possible values of p ($p=2,3,7,57$) and if $d\geq 3$, then such graphs are cycles (of length $2d+1$).
\end{theorem}

\section{Graphs of Bounded Diameter and Girth}\label{s-cycle}

In this section we will examine the trade-offs between diameter and girth, in particular for {\sc $3$-Colouring}.

Recall that Mertzios and Sprirakis~\cite{MS16} proved that $3$-{\sc Colouring} is \NP-complete for graphs of diameter~$3$ and girth~$4$ (Theorem~\ref{t-ms16}).
We extend their result in our next theorem, partially displayed in Table~\ref{t-table}. This theorem shows that there is still a large gap for which we do not know the computational complexity of {\sc $3$-Colouring} for graphs of diameter~$d$ and girth~$g$.
Note that $g\leq 2d+1$ by Lemma~\ref{l-cycle}.

\begin{theorem}\label{t-main2}
Let $d$ and $g$ be two integers with $d\geq 2$ and $g\geq 3$.
Then the following statements hold for graphs of diameter at most~$d$ and girth at least~$g$:
\begin{enumerate}
\item {\sc List Colouring} is polynomial-time solvable if $g\geq 2d+1$;
\item {\sc $3$-Colouring}  is \NP-complete if $d=3$ and $g\leq 4$;
\item {\sc $3$-Colouring} is \NP-complete  if $d\geq 4p$ and $g\leq 4p+2$ for some integer $p\geq 1$.
\end{enumerate}
\end{theorem}

\begin{proof}
{\bf 1.} This case follows
immediately 
from Theorem~\ref{t-all}.

\medskip
\noindent
{\bf 2.} This case is Theorem~\ref{t-ms16} (proven in~\cite{MS16}). 

\medskip
\noindent
{\bf 3.} We reduce $3$-{\sc{Colouring}} for graphs of girth at least $8p-3$, which is \NP-complete by Theorem~\ref{t-girth1}, to $3$-{\sc{Colouring}} for graphs of diameter at most $4p$ and girth at least $4p+2$. 
Let $G$ be a graph of girth at least $8p-3$. 
From $G$ we construct the graph $G^{\prime}$ as follows (see Figure~\ref{f-gprime} for an example):
\begin{itemize} {
\item label the vertices of $G$ $v_1$ to $v_n$;
\item for each vertex of $G$, add a new neighbour $v_{i,1}$;
\item for every two vertices $v_i$ and $v_j$ such that 
$\dist(v_i, v_j) > 2p-1$ add new vertices to form the path 
$v_iv_{i,1}v_{i,2,j}...v_{i, p+1, j}v_{j,p,i}...v_{j,1}v_j$, 
which has length~$2p+2$.
}
\end{itemize}

  \begin{center}
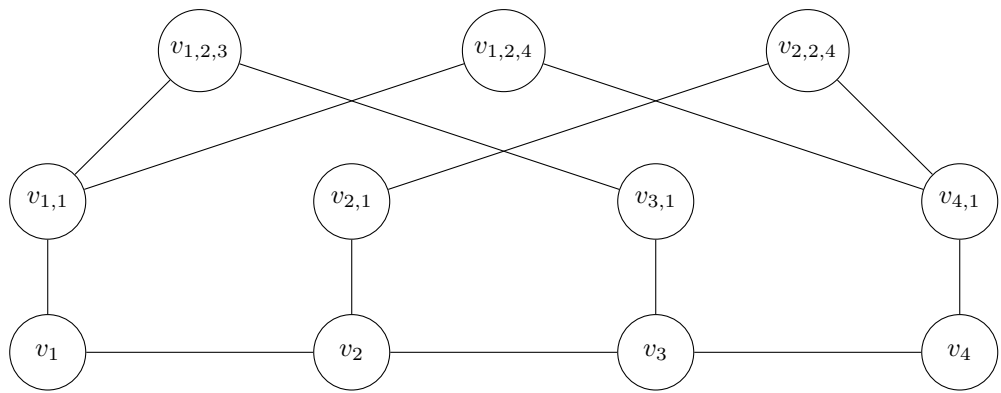
\begin{figure}[h]{
\begin{tikzpicture}[main_node/.style={circle,draw,minimum size=1cm,inner sep=3pt,scale=1]}]
	\node[main_node](v1) at (0,0){$v_1$};
	\node[main_node](v2) at (4,0) {$v_2$};
	\node[main_node](v3) at (8,0) {$v_3$};
	\node[main_node](v4) at (12,0) {$v_4$};
	\node[main_node](v5) at (0,2){$v_{1,1}$};
	\node[main_node](v6) at (4,2){$v_{2,1}$};
	\node[main_node](v7) at (8,2) {$v_{3,1}$};
	\node[main_node](v8) at (12,2) {$v_{4,1}$};
	\node[main_node](v9) at (6,4){$v_{1,2,4}$};
	\node[main_node](v10) at (2,4){$v_{1,2,3}$};
	\node[main_node](v11) at (10,4){$v_{2,2,4}$};
	
	\draw (v1)--(v2)--(v3)--(v4);
	\draw (v1)--(v5);
	\draw (v2)--(v6);
	\draw (v3)--(v7);
	\draw (v4)--(v8);
	\draw (v5)--(v9)--(v8);
	\draw (v5)--(v10)--(v7);
	\draw (v6) --(v11)--(v8);
\end{tikzpicture}
}
\caption{An example of a graph $G^{\prime}$, constructed in the proof of Theorem~\ref{t-main2}(3), for $p=1$. }\label{f-gprime}
\end{figure}
\end{center}

\noindent
First we show that $G^{\prime}$ has diameter at most {$4p$}. For any two vertices $v_i$ and $v_j$, either $\dist(v_i, v_j) \leq 2p-1$ in $G$ and thus in $G^\prime$ or we have the path $v_iv_{i,1}v_{i,2,j}...v_{i, p+1, j}v_{j,p,i}...v_{j,1}v_j$ and thus $\dist(v_i, v_j) \leq 2p+2$ in $G^\prime$. By similar arguments, we find that $\dist(v_i, v_{j,1}) \leq 2p+1$ and $\dist(v_{i,1}, v_{j,1}) \leq 2p+1$. 

Now consider two vertices $v_{a,r,b}$ and $v_{c,q,d}$ for some $2 \leq r \leq p+1$ and $2 \leq q \leq p+1$.
If $v_a=v_c$ or $v_b=v_d$, we find that $\dist(v_{a,r,b}, v_{c,q,d}) \leq 2p$.
Now suppose that $v_a,v_b,v_c,v_d$ are four distinct vertices.
If $\dist(v_a, v_c) \leq 2p-1$, then we deduce in first instance that 
\[
\begin{array}{lcl}
\dist(v_{a,r,b}, v_{c,q,d}) &\leq &r+q+2p-1\\[2mm] &\leq &(p+1)+(p+1)+(2p-1)\\[2mm] &\leq &4p+1.
\end{array}
\] 
Otherwise we have the path $v_{a,r,b}..v_{a,1}v_{a,2,c}...v_{a,p+1,c}v_{c,p,a}...v_{c,1}v_{c,2,d}...v_{c,q,d}.$ In that case, we find that $\dist(v_{a,r,b}, v_{c,q,d}) \leq (r-1)+p+p+ (q-1) \leq 4p$. In fact, if $\dist(v_{a,r,b}, v_{c,q,d})=4p+1$, then we must have $r=q=p+1$ and $\dist(v_a, v_c)=2p-1$.
Moreover, as we can also consider the pairs $(a,d)$, $(b,c)$ or $(b,d)$ instead of the pair $(a,c)$, we also find that $\dist(v_a, v_d)=\dist(v_b, v_c)=\dist(v_b, v_d)=2p-1$. In this case we have two paths $P$ and $Q$ of length $4p-2$ between $v_a$ and $v_b$, where $P$ contains $v_c$ and $Q$ contains $v_d$. 
In particular, the subpath of $P$ from $v_a$ to $v_c$ and the subpath of $P$ from $v_c$ to $v_b$ each have length $2p-1$. Then $v_d$ is not on $P$, as the existence of the vertex $v_{c,p+1, d}$ implies that $\dist(v_c, v_d) > 2p-1$. Hence, $P$ and $Q$ are two different paths.
The latter implies that $G$ has a cycle of length at most $8p-4$ which contradicts the assumption that $G$ has girth at least $8p-3$. We conclude that the diameter of $G^{\prime}$ is at most $4p$.

We now prove that $G^\prime$ has girth at least $4p+2$. Let $C$ be a cycle of $G^\prime$.
First suppose that $C$ only contains vertices of $G$. Then $C$ has length at least $8p-3$, as $G$ has girth at least $8p-3$.
If $C$ only contains vertices of $V(G^\prime)\setminus V(G)$, then by construction $C$ contains at least three vertices $v_{h,1}$, $v_{i,1}$ and $v_{j,1}$. As every path between any two such vertices has length $2p$, we find that $C$ has length at least $6p$.
Hence we may assume that $C$ contains at least one vertex of $V(G)$ and at least one vertex of $V(G^\prime)\setminus V(G)$.
All the vertices of $V(G^{\prime})\setminus V(G)$ except the vertices $v_{i,1}$ have degree~$2$. Hence, $C$ must contain the path $v_{i,1}..v_{i, p+1, j}\cdots v_{j,1}v_j$ for some $v_i$ and $v_j$ that are at distance greater than $2p-1$ in $G$. This path has length $2p+1$. If $C$ contains $v_{i,2,m}$ for some $m$ different from $j$, then $C$ also contains the path $v_{i,2,m}...v_{m,1}$, which has length $2p$. 
As $C$ must contain at least one other vertex, $C$ has length at least 
$4p+2$. We use the same arguments if $C$ contains $v_{j,2,m}$ for some $m$ different from $i$. Otherwise $C$ not only contains $v_j$ but also $v_i$. We find that $v_i$ and $v_j$ are at distance greater than $2p-1$, due to the existence of the path $v_iv_{i,1}..v_{i, p+1, j}\cdots v_{j,2}v_j$, which has length $2p+2$. Hence $C$ has length at least $(2p+2)+2p=4p+2$.

Finally, we show that $G$ is $3$-colourable if and only if $G^{\prime}$ is $3$-colourable.
If $G^\prime$ is $3$-colourable, then its induced subgraph $G$ is $3$-colourable.
Now suppose that $G$ is $3$-colourable. Let $c$ be a $3$-colouring of $G$. We first give each vertex $v_{i,1}$ a colour different from $v_i$.
Then it remains to observe that every $v_{a,r,b}$ has degree~$2$. Hence, we always have a colour available to colour such a vertex. In other words, we can extend $c$ to a $3$-colouring $c^\prime$ of $G^\prime$.
\end{proof}

\section{Generalized Claw-Free Graphs of Bounded Diameter}\label{s-polyad}

In this section we prove, among other things, our results on {\sc Colouring} and {\sc $k$-Colouring} for polyad-free graphs of bounded diameter; see also Table~\ref{fig:triad-summary}. Our first three results form the starting point of the research in this section. 

We start off with the following result for quasi-claw-free graphs.

\begin{theorem}\label{t-polynomial}
{\sc List $3$-Colouring} is polynomial-time solvable for quasi-claw-free graphs of diameter at most~$2$.
\end{theorem}

\begin{proof}
Let $G=(V,E)$ be a quasi-claw-free graph of diameter at most~$2$ that has a list $3$-assignment $L$. 
Note that if $G$ is a complete graph on more than three vertices, then $(G,L)$ is a no-instance of {\sc List $3$-Colouring}.
Hence, we may assume without loss of generality that $G$ has diameter~$2$.
This means that $G$ contains two non-adjacent vertices $u$ and $v$ whose common neighbourhood is non-empty.
Then, by definition, there exists a vertex $w$ that is a common neighbour of $u$ and $v$, such that $\{u,v\}$ dominates $N(w)$. 

We consider every possible $3$-colouring $c$ of $G[\{u,v,w\}]$ with $c(u)\in L(u)$, $c(v)\in L(v)$ and $c(w)\in L(w)$. Note that $c(u)\neq c(w)$ and $c(v)\neq c(w)$ and every neighbour $w'$ of $w$ not in $\{u,v\}$ is adjacent to at least one of $u$, $v$. Hence, such a vertex $w'$ can be coloured with at most one colour. If there is no colour available for $w'$ or the only colour available is not in $L(w)$, then we discard $c$ and try another $3$-colouring of $G[\{u,v,w\}]$. Hence, we can extend $c$ to at most one $3$-colouring of $G[N(w)\cup \{w\}]$ that respects $L$. Suppose the latter is possible. As $G$ has diameter at most~$2$, we find that $N(w)$ dominates $V$. Hence, for every uncoloured vertex of $G$ we have at most two available colours left. This means that we obtained an instance of {\sc $2$-List Colouring}. The latter is solvable in polynomial time by Theorem~\ref{t-2sat}.  As there are at most $3^3$ possible $3$-colourings of $G[\{u,v,w\}]$, we conclude that our algorithm runs in polynomial time. 
\end{proof}

For almost $K_{1,r}$-free graphs of diameter at most $d$ we can give a stronger result.
 
\begin{theorem}\label{t-constant}
For all integers $d,k,r\geq 1$, {\sc List $k$-Colouring} is constant-time solvable for almost $K_{1,r}$-free graphs of diameter at most~$d$.
\end{theorem}

\begin{proof}
Let $G=(V,E)$ be a $K_{1,r}$-free graph of diameter at most~$d$ with a list $k$-assignment~$L$. We prove that if $G$ has size larger than some constant~$\beta(k,r)$, which we determine below, then $G$ is not $k$-colourable. Hence, in that case, $(G,L)$ is a no-instance of {\sc List $k$-Colouring}.
If $|V(G)|\leq \beta(k,r)$, we can solve {\sc List $k$-Colouring} in constant time on input $(G,L)$.

Let $u\in V$. If $N(u)$ contains a clique of size $k$, then $G$ is not $k$-colourable, and hence, $(G,L)$ is a no-instance. So, by Ramsey's Theorem, we may assume that $N(u)$ contains an independent set $I(u)$ of size $(r-1)r$ if $|N(u)|\geq R(k,(r-1)r)$. By the second property of the definition of almost $K_{1,r}$-freeness, $N(u)$ contains a set $D(u)$ of size at most $r-1$ that dominates $N(u)$, and thus also dominates $I(u)$. Then, by the Pigeonhole principle, $D(u)$ contains a vertex $v$ that is adjacent to at least $r$ vertices of $D(u)$. However, now $G$ contains two adjacent centres of induced $K_{1,r}$s, namely $u$ and $v$. This violates the first property of the definition of almost $K_{1,r}$-freeness. 

From the above, we conclude that in order for $(G,L)$ to be a yes-instance of {\sc List $k$-Colouring}, every $u\in V$ must have degree less than $R(k,(r-1)r)$, so the number of vertices of $G$ must be at most $\beta(k,r)=1+ R(k,(r-1)r) + R(k,(r-1)r)^2 +\ldots+  R(k,(r-1)r)^d$.
\end{proof}

We can strengthen Theorem~\ref{t-constant} for the case where $d=2$, $k=3$, and $r=3$. For proving this result we need the following lemma.

\begin{lemma}[\cite{MPS21}]\label{c5}
{\sc List $3$-Colouring} can be solved in polynomial time for $C_5$-free graphs of diameter at most~$2$.
\end{lemma}

We can now prove the following. 

\begin{theorem}\label{t-extended0}
{\sc List $3$-Colouring} can be solved in polynomial time for 
graphs of diameter at most~$2$, in which the centre vertices of its induced claws form an independent set.
\end{theorem}

\begin{proof}
Let $G$ be a claw-free graph of diameter at most~$2$, such that the centres of the induced claws in $G$ form an independent set. Let $L$ be a list $3$-assignment of $G$. By Lemma~\ref{c5} we may assume that $G$ has an induced cycle $C$  on five vertices $u_1,\ldots,u_5$ in that order. We consider each possible $3$-colouring~$c$ of $C$ that respects $L$. For each vertex $v\in V$ that has a neighbour on $C$, we remove all colours from $L(v)$ that are used on $N(v)\cap V(C)$. If afterwards $L(v)$ has size~$0$ for some $v\in V$, then we discard $c$. Suppose that this does not happen. If $L(v)$ has size~$1$ for some $v$ not on $C$, then we let $c(v)$ be the colour in $L(v)$, and we remove $c(v)$ from the list of each uncoloured neighbour of $v$.

We claim that after applying the above procedure exhaustively, every vertex of $G$ has a list of size at most~$2$. For a contradiction, assume that $x$ is a vertex with $L(x)=\{1,2,3\}$. Then $x$ is not on $C$ and $x$ is not a neighbour of a vertex of $C$ either. As $G$ has diameter~$2$, we find that $x$ has a neighbour $v_1$ adjacent to $u_1$ and a neighbour $v_2$ adjacent to $u_2$. 
If $v_1$ is adjacent to $u_5$ or $u_2$, then $v_1$ must have been given a unique colour (as $u_1,u_2$, and likewise $u_1,u_5$, are consecutive vertices on $C$ and thus they are not coloured alike). By the same argument, $v_2$ is not adjacent to $u_1$ or $u_3$.
In particular, the above implies that $v_1$ and $v_2$ are two distinct vertices. We now observe that $u_1$ and $u_2$ are centres of induced claws with set of leaves $\{u_5,u_2,v_1\}$ and $\{u_1,u_3,v_2\}$, respectively. This is another contradiction.

From the above, we conclude that we have constructed an instance of {\sc $2$-List Colouring}. The latter is solvable in polynomial time by Theorem~\ref{t-2sat}.  As there are at most $3^5$ possible $3$-colourings of $C$, we conclude that our algorithm runs in polynomial time. 
\end{proof}

Our next result shows that we cannot generalize Theorem~\ref{t-extended0} to graphs of diameter at most~$d$ for every $d$.

\begin{theorem}\label{t-extended}
The {\sc $3$-Colouring} problem is \NP-complete for
graphs of diameter at most~$4$, in which the centre vertices of its induced claws form an independent set.
\end{theorem}

\begin{proof}
We reduce from $3$-{\sc Colouring} restricted to claw-free graphs. This problem is \NP-complete by Theorem~\ref{t-claw}. 
Let $G$ be a claw-free graph. Let $S$ be a maximal independent set of $G$. Note that we can find such a set~$S$ in polynomial time by a greedy approach.
For every pair of vertices $u,v$ that belong to $S$, we introduce a new vertex $x_{uv}$, and we make $x_{uv}$ adjacent only to  $u$ and $v$. Let $G'$ be the resulting graph. We note that $G'$ has diameter at most~$4$. Moreover, every centre of an induced claw in $G'$ belongs to the independent set $S$. It remains to observe that $G$ has a $3$-colouring if and only if $G'$ has a $3$-colouring.
\end{proof}

\noindent
We note that for locally claw-free graphs, which form a superclass of the class of graphs in Theorem~\ref{t-extended}, a stronger result holds. Namely, as triangle-free graphs are locally claw-free, {\sc $3$-Colouring} is \NP-complete for locally claw-free graphs of diameter~$3$, due to Theorem~\ref{t-ms16}. 

\medskip
\noindent
We now return to Theorem~\ref{t-constant} again.
If $k$ is not part of the input, Theorem~\ref{t-constant} no longer holds. This is shown by our next theorem.
In this theorem we assume that $H\not\ssi P_1+P_3$ and $H\not \ssi P_4$, as in those cases {\sc Colouring} is polynomial-time solvable for all $H$-free graphs due to Theorem~\ref{t-dichotomy}.
Note that Theorem~\ref{t-col} covers all remaining cases except the case where $H=K_{1,3}$.

\begin{theorem}\label{t-col}
Let $H$ be a graph with $H\not\ssi P_1+P_3$ and $H\not \ssi P_4$ and $d$ be an integer. Then {\sc Colouring} for $H$-free graphs of diameter at most~$d$ is
\begin{enumerate}
\item \NP-complete if $H$ has no dominating vertex $u$ such that $H-u\ssi P_1+P_3$ or $H-u\ssi P_4$ and $d\geq 2$;
\item \NP-complete if $H\neq K_{1,3}$ and $H$ has a dominating vertex $u$ such that $H-u\ssi P_1+P_3$ or $H-u\ssi P_4$ and $d\geq 3$.
\end{enumerate}
\end{theorem}

\begin{proof}
{\bf 1.} Let $H$ have no dominating vertex $u$ such that $H-u\ssi P_1+P_3$ or $H-u\ssi P_4$. We define $H'$ as $H-u$ if $H$ has a dominating vertex~$u$ and as $H$ itself otherwise. By construction, $H'\not\ssi P_1+P_3$ and $H'\not\ssi P_4$. Hence, {\sc Colouring} is \NP-complete for $H'$-free graphs due to Theorem~\ref{t-dichotomy}. Let $G$ be an $H'$-free graph. Add a dominating vertex to $G$. The new graph $G'$ has diameter~2 and is $H$-free. Moreover, $G$ is $k$-colourable if and only if $G'$ is $(k+1)$-colourable.

\medskip
\noindent
{\bf 2.}  Let $H\neq K_{1,3}$ have a dominating vertex $u$ such that $H-u\ssi P_1+P_3$ or $H-u\ssi P_4$.
Then $H$ cannot be a forest, as in that case $H$ would be in $\{P_1,P_2,P_3,K_{1,3}\}$.
Hence, $H$ has an induced cycle $C_r$ for some $r\geq 3$.
If $r=3$, then $3$-{\sc{Colouring}} is \NP-complete for $H$-free graphs of diameter~3, as it is so for $C_3$-free graphs of diameter~3 due to Theorem~\ref{t-ms16}. If $r\geq 4$, then {\sc Colouring} is \NP-complete even for $H$-free graphs of diameter~$2$, as it is so for $C_r$-free graphs of diameter~$2$ due to {\bf 1}. 
\end{proof}

It is a natural question whether we can extend Theorem~\ref{t-constant} to $H$-free graphs of diameter~$d$, where $H$ is a slightly larger tree than a star~$K_{1,r}$. The first interesting case is where $H$ is an $\ell$-subdivided star $K_{1,r}^\ell$ for some integer $\ell\geq 1$ and $r\geq 3$. We prove a number of results for various values of $d$,$k$ and $\ell$.
We exclude the cases that are tractable in general, namely where $d=1$; or $k\leq 2$; or $\ell = 1$ and $r\leq 2$; the latter case corresponds to the case where $H=K_{1,2}^1=P_4$, so we can use Theorem~\ref{t-dichotomy}. 
We also observe that for $k\geq 4$ all interesting cases are \NP-complete due to Theorem~\ref{t-claw} (see also Case~4 of Theorem~\ref{s112}).
However, for $k=3$ the situation is less clear.

For our results we need Lemma~\ref{c5} and the following lemma.

\begin{lemma}\label{l-cycle}
Let $d\geq 1$.
Let $G$ be a graph of diameter~$d$ that is not a tree.
If $G$ is bipartite, then the girth of $G$ is at most $2d$.
If $G$ is non-bipartite, then the girth of $G$ is at most $2d+1$ and $G$ contains an odd cycle of length at most $2d+1$.
\end{lemma}

\begin{proof}
As $G$ is not a tree and $G$ is connected, $G$ must contain a cycle $C$. Suppose that $C$ has 
length at least $2d+2$. Since $G$ has diameter~$d$, there exists a path $P$ of length at most~$d$ in $G$ between any two vertices $u$ and $v$ at distance $d+1$ in $C$. The vertices of $P$, together with the vertices of the path of length at most $d+1$ between $u$ and $v$ on $C$, induce a subgraph of $G$ that contains an induced cycle $C'$ of length at most~$2d+1$.
Hence $G$ has girth at most $2d+1$.
If $G$ is bipartite, then $C'$ has length at most $2d$, and thus $G$ has girth at most $2d$.

We now assume that $G$ is a non-bipartite graph. Then $G$ must contain an odd cycle~$C$.
Suppose that $C$ has odd length at least $2d+3$. As before, there exists a path~$P$ of length at most $d$ in $G$ between any two vertices $u$ and $v$ at distance $d+1$ in $C$. If the cycle formed by the vertices of $P$ together with the vertices of the path of length $d+1$ between $u$ and $v$ in $C$ is odd we have an odd cycle of length at most $2d+1$. Otherwise we consider the longer path between $u$ and $v$ in $C$. The vertices of this path together with the vertices of $P$ induce an odd cycle shorter than $C$. By repeating this process we obtain an odd cycle of length at most $2d+1$.
\end{proof}

\noindent
We can now state and prove the following result.

\begin{theorem}\label{s112}
Let $d,k,\ell, r$ be four integers with $d\geq 2$, $k\geq 3$, $\ell\geq 1$ and $r\geq 3$. Then for $K_{1,r}^\ell$-free graphs of diameter at most~$d$, the following holds:
\begin{enumerate}
\item {\sc List $k$-Colouring} is polynomial-time solvable if $d\geq 2$, $k=3$, $\ell=1$ and $r=3$
\item{\sc List $k$-Colouring} is polynomial-time solvable if $d=2$, $k=3$, $\ell=2$ and $r\geq 3$
\item {\sc $k$-Colouring} is \NP-complete if $d\geq 4$, $k=3$, $\ell\geq 3$ and $r\geq 4$
\item {\sc $k$-Colouring} is \NP-complete if $d\geq 2$, $k\geq 4$, $\ell\geq 1$ and $r\geq 3$.
\end{enumerate}
\end{theorem}

\begin{proof}
{\bf 1.} Recall that $K_{1,3}^1$ is the chair $S_{1,1,2}$. 
Let $(G,L)$ be an instance of {\sc List $3$-Colouring}, where $G$ is a chair-free graph of diameter~$d$ for some $d\geq 2$.

First suppose that $G$ is a tree. We consider a leaf $u$. If $L(u)=\emptyset$, then $(G,L)$ is a no-instance.
 If $|L(u)|=1$, then we assign $u$ the unique colour of $L(u)$ and remove that colour from the list of the parent of $u$. 
 If $|L(u)|\geq 2$, then $(G,L)$ is a yes-instance if and only if $(G',L')$ is a yes-instance, where $G'=G-u$ and $L'$ is the restriction of $L$ to $V(G)\setminus \{u\}$.
Hence, we can determine in polynomial time if $G$ has a colouring that respects $L$. 
From now on assume that $G$ is not a tree.

We check in $O(n^4)$ time if $G$ has a $K_4$. If so, then $G$ is not $3$-colourable and thus $G$ has no colouring respecting $L$.
From now on we assume that $G$ is not a tree and that $G$ is $K_4$-free. As $G$ is not a tree and $G$ is connected, $G$ contains an induced cycle of length at most $2d+1$ by Lemma~\ref{l-cycle}. We can find a largest induced cycle~$C$ of length at most $2d+1$ in $O(n^{2d+1})$ time.
Let $|V(C)|=p$.
We write $N_0=V(C)=\{x_1,x_2,\ldots,x_p\}$ and for $i\geq 1$, $N_i=N(N_{i-1})\setminus N_{i-2}$. So the sets $N_i$ partition $V(G)$, and the distance of a vertex $u\in N_i$ to $N_0$ is $i$. 

\begin{figure} [h]
	\resizebox{10cm}{!} {
		\begin{tikzpicture}[main_node/.style={circle,draw,minimum size=1cm,inner sep=3pt]}]
			\node[main_node](x1) at (0,0){$x_1$};
			\node[main_node](x2) at (2,0){$x_2$};
			\node[main_node](x3) at (4,0){$x_3$};
			\node[main_node](x4) at (6,0){$x_4$};
			\node[main_node](x5) at (8,0){$x_5$};
			
			\draw (x1)--(x2)--(x3)--(x4)--(x5);
			\draw (x1) to[out=30, in=150] (x5);
			
			\draw[thick, dash dot] (-4,-1)--(10,-1);
			
			\node[main_node](y1) at (0,-2){y};
			
			\draw(x1)--(y1);
			\draw(x2)--(y1);
			
			\draw[thick, dash dot] (-4,-3)--(10,-3);
			
			\node[main_node](z1) at (0,-4){z};
			
			\draw(y1)--(z1);
			
			\draw[thick, dash dot] (-4,-5)--(10,-5);
			
			\node[main_node](v1) at (-1,-6){$w_1$};
			\node[main_node](v2) at (1,-6){$w_2$};
			
			\draw(v1)--(z1);
			\draw(v2)--(z1);
			\draw(v1)--(v2);
			
			\draw[thick, dash dot] (-2,2)--(-2,-8);
			
			\node(n0) at (-3,0){$N_0$};
			\node(n1) at (-3,-2){$N_1$};
			\node(n2) at (-3,-4){$N_2$};
			\node(n3) at (-3,-6){$N_3$};
		\end{tikzpicture}
}
		\caption{An example of a decomposition of a chair-free graph of diameter~3 into sets $N_0,\ldots, N_3$ where $p=5$ and $y\in N_1$ has two ``descendants'' in $N_3$. To prevent an induced chair, $y$ must be adjacent to exactly two (adjacent) vertices of $N_0$, and $w_1$ and $w_2$ must be adjacent to each other.}\label{chair} 
	\end{figure}
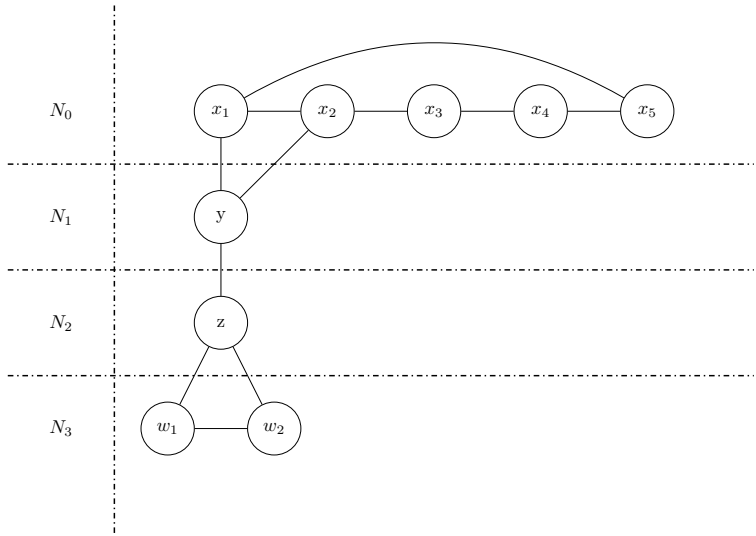
	
\medskip
\noindent
{\bf Case 1.} $4\leq p\leq 2d+1$. \\
This case is illustrated in Figure~\ref{chair}. We consider every possible $3$-colouring of $C$ that respects the restriction of $L$ to $V(C)$. Let $c$ be such a $3$-colouring.
Every vertex with two differently coloured neighbours can only be coloured with one remaining colour. We assign this unique colour to such a vertex and apply this rule as long as possible. 
This takes polynomial time. 
We discard $c$ as soon as the list of a vertex is empty. Otherwise, the remaining vertices have a list of admissible colours that either consists of two or three colours, and vertices in the latter case belong to $V(G)\setminus (N_0\cup N_1)$ (as $N(N_0)=N_1$). 

If $N_2=\emptyset$, then $V(G)=N_0\cup N_1$. Then, we obtained an instance of {\sc $2$-List Colouring}, which we can solve in linear time due to Theorem~\ref{t-2sat}. Now assume that $N_2\neq \emptyset$.
Let $z\in N_2$. Then $z$ has a neighbour $y\in N_1$, which in turn has a neighbour $x\in N_0$.
 If $y$ is adjacent to neither neighbour of $x$ on $N_0$, then $z$, $y$, $x$ and these two neighbours induce a chair in $G$, a contradiction.  Hence, $y$ must be adjacent to at least one neighbour of $x$ on $N_0$, meaning that $y$ must have received a colour by our algorithm. Consequently, $z$ must have a list of admissible colours of size at most~$2$.
 
From the above we deduce that every vertex in $N_2$ has only two available colours in its list. We now consider the vertices of $N_3$.
Let $z'\in N_3$. Then $z'$ has a neighbour $z\in N_2$, which in turn has a neighbour $y\in N_1$, which in turn has a neighbour $x\in N_0$, say $x=x_1$. If $y$ has two non-adjacent neighbours in $N_0$, then $z',z,y$ and these two non-adjacent neighbours of $y$ induce a chair in $G$, a contradiction. Combined with the fact deduced above, we conclude that $y$ must have exactly two neighbours in $N_0$ and these two neighbours must be adjacent, say $x_2$ is the other neighbour of $y$ in $N_0$.

Suppose $x_1$ and $x_2$ are both adjacent to a vertex $y'\in N_1\setminus \{y\}$ that is adjacent to a vertex in $N_2$ that has a neighbour in $N_3$. Then, just as in the case of vertex~$y$, the two vertices $x_1$ and $x_2$ are the only two neighbours of $y'$ in $N_0$. If $y$ and $y'$ are not adjacent, this means that $x_2,x_3,x_4,y,y'$ induce a chair in $G$, a contradiction. Hence $y$ and $y'$ must be adjacent. However, then $x_1,x_2,y,y'$ form a $K_4$, a contradiction. This means that every pair of adjacent vertices of $N_0$ can have at most one common neighbour in $N_1$ that is adjacent to a vertex in $N_2$ with a neighbour in $N_3$.
We already deduced that every vertex of $N_1$ with a ``descendant'' in $N_3$ has exactly two neighbours in $N_0$, which are adjacent. Hence, we conclude that the number of such vertices of $N_1$ is at most $p$. 
 
We now observe that for $i\geq 2$, every vertex in $N_i$ has at most two neighbours in $N_{i+1}$. This can be seen as follows. If $v\in N_i$ has two non-adjacent neighbours $w_1,w_2$ in $N_{i+1}$, then we pick a neighbour $u$ of $v$ in $N_{i-1}$, which has a neighbour~$t$ in $N_{i-2}$. Then $v,u,t,w_1,w_2$ induce a chair in $G$, a contradiction. Hence , the neighbourhood of every vertex in $N_i$ in $N_{i+1}$ is a clique, which must have size at most~$2$ due to the $K_4$-freeness of $G$.
As the number of vertices in $N_1$ with a ``descendant'' in $N_3$ is at most $p$, this means that there are at most $2^{i-1}p$ vertices in $N_i$ with a neighbour in $N_{i+1}$. Therefore the total number of vertices not belonging to any of the sets $N_0, N_1$ or $N_2$ is at most $\sum_{i=3}^{d} 2^{i-1}p.$
This means the total number of vertices not belonging to $N_1$ or $N_2$ is at most $$\beta(d)\,=\, \sum_{i=3}^{d}{ 2^{i-1}p}+p\, \leq\,
\sum_{i=3}^{d}{ 2^{i-1}(2d+1)}+2d+1.$$ 
Let $T_c$ be the set consisting of these vertices. We consider every possible $3$-colouring of $G[T_c]$ that respects the restriction of $L$ to $T_c$. As we already deduced that the vertices in $N_1\cup N_2$ have a list of size at most~2, for each case we obtain an instance of $2$-{\sc List Colouring}, which we can solve in linear time due to Theorem~\ref{t-2sat}. As the total number of instances we need to consider is at most $3^p \cdot 3^{\beta(d)}\leq 3^{2d+1}\cdot 3^{\beta(d)}$, our algorithm runs in polynomial time.

\medskip
\noindent
{\bf Case 2.} $p=3$.\\
As $p$ was the size of a largest induced cycle of length at most~$2d+1$ and $2d+1\geq 5$, we find that $G$ is $C_4$-free.
As $G$ is $K_4$-free, each vertex of $N_1$ is adjacent to at most two vertices of~$N_0$.
If a vertex $x\in N_0$ has two independent private neighbours~$u$ and~$v$ in $N_1$ with respect to $N_0$, then every neighbour~$w$ of $u$ in  $N_2$ must also be a neighbour of $v$ and vice versa, since $G$ is chair-free. However, this is not possible, as $x,u,w,v$ induce a $C_4$. We conclude that $u$ and $v$ must be adjacent. Therefore, as $G$ is $K_4$-free, every vertex of $N_0$ has at most two private neighbours in $N_1$, with respect to $N_0$, that have a neighbour in $N_2$.

By the same arguments as above we deduce that every two vertices of $N_0$ have at most one common neighbour in $N_1$ that is adjacent to a vertex in $N_2$. Combined with the above, we find that there at most 
 $6+3=9$ vertices in $N_1$ that have a neighbour in $N_2$. If a vertex in $N_1$ has two independent neighbours in $N_2$, then $G$ contains an induced chair, which is not possible. 
Hence the neighbourhood of a vertex in $N_1$ in $N_2$ is a clique, which has size at most~$2$ due to the $K_4$-freeness of $G$.
We conclude that  $|N_2|\leq 9\times 2=18$. 
Similarly, every vertex in $N_i$ for $i \geq 3$ has at most two neighbours in $N_{i+1}$. Therefore the number of vertices in $N_i$ for $i \geq 3$ is at most 
$18 \times 2^{i-2}$. This means that the total number of vertices outside $N_0 \cup N_1 \cup N_2$ is at most $$\beta(d)=\sum_{i=3}^{d} 18 \times 2^{i-2}.$$ 
Let $T$ be the set consisting of these vertices. We consider every possible $3$-colouring of $G[V(C)\cup T]$ that respects the restriction of $L$ to $V(C)\cup T$. For each case we obtain an instance of $2$-{\sc List Colouring}, which we can solve in linear time due to Theorem~\ref{t-2sat}. As the total number of instances we need to consider is at most $3^3 \times 3^{\beta(d)}$, our algorithm runs in polynomial time.

\medskip
\noindent
{\bf 2.} Let $(G,L)$ be an instance of {\sc List $3$-Colouring}, where $G$ is a $K_{1,r}^2$-free graph of diameter at most~$2$ for some $r\geq 3$. We first check in $O(n^4)$ time if $G$ is $K_4$-free. If not, then $G$ is not $3$-colourable, and thus $(G,L)$ is a no-instance. We then check in $O(n^5)$ time if $G$ has an induced $C_5$.
If $G$ is $C_5$-free, then we use Lemma~\ref{c5}. 
From now on, suppose that $G$ is $K_4$-free and that $G$ contains an induced cycle $C$ of length~$5$, say on vertices $x_1,\ldots,x_5$ in that order.
We write $N_0=V(C)=\{x_1,\ldots,x_5\}$, $N_1=N(V(C))$ and $N_2=V(G)\setminus (N_0\cup N_1)$.

Let $N_2'$ be the set of vertices in $N_2$ that are adjacent to some vertex in $N_1$ that is a private neighbour of some vertex in $N_0$ with respect to $N_0$.  
As $G$ is $K_4$-free, the private neighbourhood $P(x_i)$ of each vertex $x_i\in N_0$ with respect to $N_0$ does not contain a clique of size $3$. Moreover, if $P(x_i)$ contains an independent set $I$ of size $r-1$ for some $i\in \{1,\ldots,5\}$, then $I\cup \{x_i,x_{i+1},x_{i+2},x_{i+3}\}$ induces a $K_{1,r}^2$, which is not possible. Now let $v\in P(x_i)$ for some $i\in \{1,\ldots 5\}$, say $i=1$. As $G$ is $K_4$-free, the set $N(v)\cap N_2$ does not contain a clique of size $3$. Moreover, if $N(v)\cap N_2$ contains an independent set $I'$ of size $r-1$, then $I'\cup \{v,x_1,x_2,x_3,\}$ induces a $K_{1,r}^2$, which is not possible. 
 Hence, $|N(v)\cap N_2|\leq  R(3,r-1)$ by Ramsey's Theorem. We conclude that 
 $$|N_2'|\leq 5R(3,r-1)^2.$$
We now consider all possible $3$-colourings of $C$ that respect the restriction of $L$ to $V(C)$. Let $c$ be such a $3$-colouring. We assume without loss of generality that $c(x_1)=c(x_3)=1$, $c(x_2)=c(x_4)=2$ and $c(x_5)=3$. 
Moreover, every vertex that has two differently coloured neighbours can only be coloured with one remaining colour. We assign this unique colour to such a vertex and apply this rule as long as possible. This takes polynomial time. 
We discard~$c$ as soon as the list of a vertex is empty.
The remaining vertices have a list of admissible colours that either consists of two or three colours, and vertices in the latter case must belong to $N_2$ (as $N(N_0)=N_1$).

Let $T_c$ be the set of vertices in $N_2$ that still have a list of size~$3$. We will prove that $T_c\subseteq N_2'$. 
Let $u\in T_c$. As $G$ has diameter~$2$, we find that $u$ has a neighbour $v$ adjacent to $x_5$. Then $v$ cannot be adjacent to any of $x_1,\ldots,x_4$, as otherwise $v$ would have a unique colour and $u$ would not be in $T_c$. Hence, $v$ is a private neighbour of $x_5$ with respect to $N_0$. We conclude that all vertices in $T_c$ belong to $N_2'$, which implies that $|T_c|\leq |N_2'|\leq 5R(3,r-1)^2$.

We now consider every possible $3$-colouring of $G[T_c]$ that respects the restriction of $L$ to~$T_c$. Then all uncoloured vertices have a list of size at most~$2$. In other words, we created an instance of {\sc $2$-List Colouring}, which we solve in linear time using Theorem~\ref{t-2sat}. As the number of $3$-colourings of $C$ is at most $3^5$ and for each $3$-colouring $c$ of $C$ the number of $3$-colourings of $G[T_c]$ is at most
$3^{5R(3,r-1)^2}$, the total running time of our algorithm is polynomial.

\medskip
\noindent
{\bf 3.} We consider the standard reduction from the \NP-complete problem NAE $3$-SAT~\cite{Sc78}, where each variable appears in at most three clauses and each literal appears in at most two. Given a CNF formula $\phi$, we construct the graph $G$ as follows (see also Figure~\ref{f-figgg}):

\begin{itemize} 
    \item Add a literal vertex $v_i$ for each positive literal $x_i$ and a literal vertex $v_i'$ for its negation.
    \item Add an edge between each literal vertex and its negation.
    \item Add a vertex $z$ adjacent to every literal vertex.
    \item For each clause $C_i$ add a triangle $T_i$ with vertices $c_{i_1}, c_{i_2}, c_{i_3}$.
    \item Fix an arbitrary order of the literals of $C_i$, $x_{i_1}, x_{i_2}, x_{i_3}$ and add an edge $x_{i_j}c_{i_j}$.
\end{itemize}

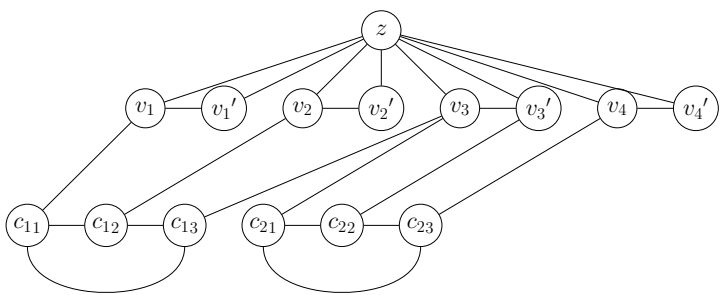
\begin{figure}[h]
\resizebox{9.5cm}{!}{
\begin{tikzpicture}[main_node/.style={circle,draw,minimum size=1cm,inner sep=3pt]}]
\node[main_node](z) at (2,0){\LARGE$z$};
\node[main_node](v_1) at (-4,-2){\LARGE$v_1$};
\node[main_node](v1') at (-2,-2){\LARGE${v_1}^{\prime}$};
\node[main_node](v_2) at (0,-2){\LARGE$v_2$};
\node[main_node](v_2') at (2,-2){\LARGE${v_2}^{\prime}$};
\node[main_node](v_3) at (4,-2) {\LARGE$v_3$};
\node[main_node](v_3') at (6,-2) {\LARGE${v_3}^{\prime}$};
\node[main_node](v4) at (8,-2) {\LARGE$v_4$};
\node[main_node](v4') at (10,-2){\LARGE${v_4}^{\prime}$};
\node[main_node](c11) at (-5,-5) {\LARGE${c_1}_2$};
\node[main_node](c12) at (-7, -5) {\LARGE${c_1}_1$};
\node[main_node](c13) at (-3,-5) {\LARGE${c_1}_3$};
\node[main_node](c21) at (1,-5) {\LARGE${c_2}_2$};
\node[main_node](c22) at (-1,-5) {\LARGE${c_2}_1$};
\node[main_node](c23) at (3,-5) {\LARGE${c_2}_3$};
\draw(z)--(v4);
\draw(z)--(v4');
\draw(z)--(v_1);
\draw(z)--(v1');
\draw(z)--(v_2);
\draw(z)--(v_2');
\draw(z)--(v_3);
\draw(z)--(v_3');
\draw(v_1)--(c12);
\draw(v_2)--(c11);
\draw(v_3)--(c22);
\draw(v_3)--(c13);
\draw(v_3')--(c21);
\draw(v4)--(c23);
\draw[](c11)--(c12);
\draw[](c11)--(c13);
\draw[](c12) to [out=270, in=270](c13);
\draw[](c21)--(c22);
\draw[](c21)--(c23);
\draw[](c22)to [out=270, in=270](c23);
\draw[](v_1)--(v1');
\draw[](v_2)--(v_2');
\draw[](v_3)--(v_3');
\draw[](v4)--(v4');
\end{tikzpicture}
}
\caption{An example of a graph $G$ in the reduction from NAE $3$-SAT to {\sc $3$-Colouring} with clauses $C_1=x_1 \lor x_2 \lor x_3$ and $C_2=x_3\lor \neg x_3 \lor x_4$.}\label{f-figgg}
\end{figure}

For the sake of completeness we give the arguments for the known reduction.
Given a $3$-colouring of $G$, assume $z$ is assigned colour $1$. Then each literal vertex is assigned either colour $2$ or colour $3$. If, for some clause $C_i$, the vertices $x_{i_1}, x_{i_2}$ and $x_{i,3}$ are all assigned the same colour, then $T_i$ cannot be coloured. Therefore, if we set literals whose vertices are coloured with colour $2$ to be true and those coloured with colour $3$ to be false, each clause must contain at least one true literal and at least one false literal.

If $\phi$ is satisfiable then we can colour vertex $z$ with colour $1$, each true literal with colour $2$ and each false literal with colour $3$. Then, since each clause has at least one true literal and at least one false literal, each triangle has neighbours in two different colours. This implies that each triangle is $3$-colourable. Therefore $G$ is $3$-colourable if and only if $\phi$ is satisfiable.

We next show that $G$ has diameter at most $4$. First note that any literal vertex is adjacent to $z$ and any clause vertex is adjacent to some literal vertex so any vertex is at distance at most $2$ from $z$. Therefore any two vertices are at distance at most $4$.

Finally we show that $G$ is $K_{1,4}^3$-free. Any literal vertex has degree at most $4$ since it appears in at most two clauses. However it has at most three independent neighbours since its negation is adjacent to $z$. Each clause vertex has at most three neighbours so the only vertex with four independent neighbours is $d$. The longest induced path including $z$ has length at most~$4$ since any such path contains at most one literal vertex and at most two vertices of any triangle. Therefore $G$ is $K_{1,4}^3$-free.

\medskip
\noindent
{\bf 4.} This follows from Theorem~\ref{t-claw}. Let $k^*\geq 3$. We take a claw-free graph $G$ and add a dominating vertex to it.
The new graph $G'$ has diameter at most~$2$ and is $K_{1,3}^1$-free. Let $k=k^*+1\geq 4$. Then $G$ is $k^*$-colourable if and only if $G'$ is $k$-colourable.
\end{proof}

Subdividing two edges of the claw yields another interesting case, namely where  $H=S_{1,2,2}$.
For $k\geq 4$, Theorem~\ref{s112} tells us that {\sc $k$-Colouring} is \NP-complete for $S_{1,2,2}$-free graphs of diameter~2.
For $k=3$, we could only prove polynomial-time solvability if $d=2$. 

\begin{theorem}\label{s122}
{\sc List $3$-Colouring} can be solved in polynomial time for $S_{1,2,2}$-free graphs of diameter at most~$2$.
\end{theorem}

\begin{proof}
Let $(G,L)$ be an instance of {\sc List $3$-Colouring}, where $G$ is an $S_{1,2,2}$-free graph of diameter at most~$2$.
We first check in $O(n^5)$ time if $G$ has an induced $C_5$.
If $G$ is $C_5$-free, then we use Lemma~\ref{c5}. Suppose $G$ contains an induced cycle $C$ of length~$5$, say on vertices $x_1,\ldots,x_5$ in that order.
We write $N_0=V(C)=\{x_1,\ldots,x_5\}$, $N_1=N(V(C))$ and $N_2=V(G)\setminus (N_0\cup N_1)$.
As $G$ has diameter~$2$, for every $i\in \{1,2,3\}$,  every vertex in $N_2$ has a neighbour in $N_1$ that is adjacent to $x_i$. 

We let $T$ consist of all vertices of $N_2$ that have a neighbour in $N_1$ that is adjacent to two adjacent vertices of $N_0$. 
So after colouring the five vertices of $N_0$, there are at most two option left for colouring a vertex of $T$.
We claim that $N_2=T$. In order to see this, let $u\in N_2$. As $G$ has diameter~2, we find that $u$ must have a neighbour $v\in N_1$ adjacent to a vertex of $N_0$, say $x_1$. Then $v$ is not adjacent to $x_5$ or $x_2$. If $v$ is not adjacent to $x_3$ either, then the vertices
$x_1,x_5,x_2,x_3,v,u$ induce a $S_{1,2,2}$ with centre~$x_1$, a contradiction. So $v$ must be adjacent to $x_3$, meaning $v$ is not adjacent to $x_4$. However, now $x_3,x_2,x_4,x_5,v,u$ induce a $S_{1,2,2}$ with centre~$x_3$, another contradiction.

We now consider every possible $3$-colouring of $C$  that respects the restriction of $L$ to $V(C)$. We observe that every vertex of $N_1$ can only be coloured with at most two possible colours and that after propagation, every uncoloured vertex of $N_2$ can only be coloured with two possible colours as well (as $T=N_2$). Then it remains to solve
an instance of {\sc $2$-List Colouring}, which takes linear time by Theorem~\ref{t-2sat}. As we need to do this at most $3^5$ times, the total running time of our algorithm is polynomial.
\end{proof}

\section{Conclusions}\label{s-con}

We proved a number of new results for {\sc Colouring} and {\sc $k$-Colouring} for polyad-free graphs of bounded diameter and for graphs of bounded diameter and girth. In particular we identified and narrowed a number of complexity gaps. This leads us to some natural open problems.
Open Problems~1 and~2 follow from Theorem~\ref{t-col}. Open Problems~3 and~4 come from Theorem~\ref{s112}. We note that $K_{1,3}^2=S_{1,1,3}$. 
We also note that the gadget Mertzios and Spirakis~\cite{MS16} used in the proof of Theorem~\ref{t-ms16} contains induced polyads with an arbitrarily large diameter and an arbitrarily large number of leaves.
So in order to reduce the diameter in Theorem~\ref{s112}:3 from $d=4$ to $d=3$ (if possible) one must find an alternative \NP-hardness proof for $3$-{\sc Colouring} restricted to graphs of diameter~$3$.
Open Problem~5 stems from Theorem~\ref{s122}.
Recall that determining the complexity of $3$-{\sc Colouring} for graphs of diameter~2 is still wide open. This question is covered by Open Problem~6.
Open Problems~7 and~8 stem from Theorem~\ref{t-polynomial}. Finally, Open Problem~9 is closely related to Theorem~\ref{t-ms16} and Open Problem~2, as triangle-free graphs form a subclass of locally claw-free graphs.

\begin{oproblem}
Does there exist an integer~$d$ such that {\sc Colouring} is \NP-complete for $K_{1,3}$-free graphs of diameter~$d$?
\end{oproblem}

\begin{oproblem}
What is the complexity of {\sc Colouring} for $C_3$-free graphs of diameter~$2$, or equivalently, graphs of diameter~$2$ and girth at least~$4$?
\end{oproblem}

\begin{oproblem}
What are the complexities of {\sc $3$-Colouring} for $K_{1,4}^1$-free graphs of diameter~$3$ and for $K_{1,3}^2$-free graphs of diameter~$3$?
\end{oproblem}

\begin{oproblem}
Does there exist a polyad $S$ such that {\sc $3$-Colouring} is \NP-complete for $S$-free graphs of diameter~$3$?
\end{oproblem}

\begin{oproblem}
Do there exist integers $d,h,i,j$ such that $3$-{\sc Colouring} is \NP-complete for $S_{h,i,j}$-free graphs of diameter~$d$?
\end{oproblem}

\begin{oproblem}
What is the complexity of the open cases in Table~\ref{t-table} and in particular of $3$-{\sc Colouring} for graphs of diameter~$2$ and for
$C_3$-free graphs of diameter~$2$?
\end{oproblem}

\begin{oproblem}
What is the complexity of {\sc $3$-Colouring} for quasi-claw-free graphs of diameter at most~$3$?
\end{oproblem}

\begin{oproblem}
Does there exists an integer $d$ such that {\sc $3$-Colouring} is \NP-complete for quasi-claw-free graphs of diameter at most~$d$?
\end{oproblem}

\begin{oproblem}
What is the complexity of {\sc $3$-Colouring} for locally claw-free graphs of diameter at most~$2$?
\end{oproblem}


\begin{thebibliography}{10}

\bibitem{Ai98}
Ahmed Ainouche.
\newblock Quasi-claw-free graphs.
\newblock {\em Discrete Mathematics}, 179:13--26, 1998.

\bibitem{Al93}
Noga Alon.
\newblock Restricted colorings of graphs.
\newblock {\em Surveys in combinatorics, London Mathematical Society Lecture
  Note Series}, 187:1--33, 1993.

\bibitem{BKM12}
Manuel Bodirsky, Jan K{\'{a}}ra, and Barnaby Martin.
\newblock The complexity of surjective homomorphism problems - a survey.
\newblock {\em Discrete Applied Mathematics}, 160:1680--1690, 2012.

\bibitem{BCMSSZ18}
Flavia Bonomo, Maria Chudnovsky, Peter Maceli, Oliver Schaudt, Maya Stein, and
  Mingxian Zhong.
\newblock Three-coloring and list three-coloring of graphs without induced
  paths on seven vertices.
\newblock {\em Combinatorica}, 38:779--801, 2018.

\bibitem{BGMPS21b}
Christoph Brause, Petr~A. Golovach, Barnaby Martin, Dani{\"{e}}l Paulusma, and
  Siani Smith.
\newblock Acyclic, star, and injective colouring: Bounding the diameter.
\newblock {\em Proc. WG 2021, LNCS}, 12911:336--348, 2021.

\bibitem{BGMPS21}
Christoph Brause, Petr~A. Golovach, Barnaby Martin, Dani{\"{e}}l Paulusma, and
  Siani Smith.
\newblock Partitioning ${H}$-free graphs of bounded diameter.
\newblock {\em Proc. ISAAC 2021, LIPIcs}, to appear, 2021.

\bibitem{BFGP13}
Hajo Broersma, Fedor~V. Fomin, Petr~A. Golovach, and Dani\"el Paulusma.
\newblock Three complexity results on coloring {$P_k$}-free graphs.
\newblock {\em European Journal of Combinatorics}, 34:609--619, 2013.

\bibitem{CC06}
Miroslav Chleb{\'{\i}}k and Janka Chleb{\'{\i}}kov{\'{a}}.
\newblock Hard coloring problems in low degree planar bipartite graphs.
\newblock {\em Discrete Applied Mathematics}, 154:1960--1965, 2006.

\bibitem{C14}
Maria Chudnovsky.
\newblock Coloring graphs with forbidden induced subgraphs.
\newblock {\em Proc. ICM 2014}, IV:291--302, 2014.

\bibitem{CHSZ18}
Maria Chudnovsky, Shenwei Huang, Sophie Spirkl, and Mingxian Zhong.
\newblock List-three-coloring graphs with no induced ${P}_6+r{P}_3$.
\newblock {\em Algorithmica}, 83:216--251, 2021.

\bibitem{CSZ19}
Maria Chudnovsky, Sophie Spirkl, and Mingxian Zhong.
\newblock Four-coloring ${P}_6$-free graphs.
\newblock {\em Proc. SODA 2019}, pages 1239--1256, 2019.

\bibitem{DDJP}
Konrad~Kazimierz Dabrowski, Fran{\c{c}}ois Dross, Matthew Johnson, and
  Dani{\"{e}}l Paulusma.
\newblock Filling the complexity gaps for colouring planar and bounded degree
  graphs.
\newblock {\em Journal of Graph Theory}, 92:377--393, 2019.

\bibitem{D73}
R.~M. Damerell.
\newblock On {M}oore graphs.
\newblock {\em Proc. Cambridge Philos. Soc}, 74:227--236, 1973.

\bibitem{DPR21}
Michal D\k{e}bski, Marta Piecyk, and Pawe{\l} Rz{\k{a}}\.{z}ewski.
\newblock Faster $3$-coloring of small-diameter graphs.
\newblock {\em Proc. ESA 2021, LIPIcs}, 204:37:1--37:15, 2021.

\bibitem{Ed86}
Keith Edwards.
\newblock The complexity of colouring problems on dense graphs.
\newblock {\em Theoretical Computer Science}, 43:337--343, 1986.

\bibitem{EHK98}
Thomas Emden-Weinert, Stefan Hougardy, and Bernd Kreuter.
\newblock Uniquely colourable graphs and the hardness of colouring graphs of
  large girth.
\newblock {\em Combinatorics, Probability and Computing}, 7:375--386, 1998.

\bibitem{GJPS17}
Petr~A. Golovach, Matthew Johnson, Dani\"el Paulusma, and Jian Song.
\newblock A survey on the computational complexity of colouring graphs with
  forbidden subgraphs.
\newblock {\em Journal of Graph Theory}, 84:331--363, 2017.

\bibitem{GPS14}
Petr~A. Golovach, Dani{\"{e}}l Paulusma, and Jian Song.
\newblock Coloring graphs without short cycles and long induced paths.
\newblock {\em Discrete Applied Mathematics}, 167:107--120, 2014.

\bibitem{GOPSSS18}
Carla Groenland, Karolina Okrasa, Pawel Rzążewski, Alex Scott, Paul Seymour,
  and Sophie Spirkl.
\newblock ${H}$-colouring ${P}_t$-free graphs in subexponential time.
\newblock {\em Discrete Applied Mathematics}, 267:184--189, 2019.

\bibitem{GLS84}
Martin Gr\"otschel, L\'aszl\'o Lov\'asz, and Alexander Schrijver.
\newblock Polynomial algorithms for perfect graphs.
\newblock {\em Annals of Discrete Mathematics}, 21:325--356, 1984.

\bibitem{HLS}
Sepehr Hajebi, Yanjia Li, and Sophie Spirkl.
\newblock Complexity dichotomy for {L}ist-$5$-{C}oloring with a forbidden
  induced subgraph.
\newblock {\em CoRR}, abs/2105.01787, 2021.

\bibitem{HKLSS10}
Ch\'inh~T. Ho\`ang, Marcin Kami\'nski, Vadim~V. Lozin, Joe Sawada, and Xiao
  Shu.
\newblock Deciding $k$-colorability of {$P_5$}-free graphs in polynomial time.
\newblock {\em Algorithmica}, 57:74--81, 2010.

\bibitem{HS60}
Alan~J. Hoffman and Robert~R. Singleton.
\newblock On {M}oore graphs with diameter $2$ and $3$.
\newblock {\em IBM Journal of Research and Development}, 5:497--504, 1960.

\bibitem{Ho81}
Ian Holyer.
\newblock The {N}{P}-completeness of edge-coloring.
\newblock {\em {SIAM} Journal on Computing}, 10:718--720, 1981.

\bibitem{Hu16}
Shenwei Huang.
\newblock Improved complexity results on $k$-coloring {$P_t$-free} graphs.
\newblock {\em European Journal of Combinatorics}, 51:336--346, 2016.

\bibitem{JT95}
Tommy~R. Jensen and Bjarne Toft.
\newblock {\em Graph coloring problems}.
\newblock John Wiley \& Sons, 1995.

\bibitem{KMMNPS18}
Tereza Klimo\v{s}ov\'a, Josef Mal\'ik, Tom\'a\v{s} Masa\v{r}\'ik, Jana
  Novotn\'a, Dani\"el Paulusma, and Veronika Sl\'ivov\'a.
\newblock Colouring $({P}_r+{P}_s)$-free graphs.
\newblock {\em Algorithmica}, 82:1833--1858, 2020.

\bibitem{KKTW01}
Daniel Kr{\'a}l', Jan Kratochv\'{\i}l, {\relax Zs}olt Tuza, and Gerhard~J.
  Woeginger.
\newblock Complexity of coloring graphs without forbidden induced subgraphs.
\newblock {\em Proc. WG 2001, LNCS}, 2204:254--262, 2001.

\bibitem{KTV99}
Jan Kratochv\'il, {\relax Zs}olt Tuza, and Margit Voigt.
\newblock New trends in the theory of graph colorings: choosability and list
  coloring.
\newblock {\em Proc. DIMATIA-DIMACS Conference}, 49:183--197, 1999.

\bibitem{LG83}
Daniel Leven and Zvi Galil.
\newblock {NP} completeness of finding the chromatic index of regular graphs.
\newblock {\em Journal of Algorithms}, 4:35--44, 1983.

\bibitem{Lo73}
L\'aszl\'o Lov\'asz.
\newblock Coverings and coloring of hypergraphs.
\newblock {\em Congr. Numer.}, VIII:3--12, 1973.

\bibitem{LM17}
Vadim~V. Lozin and Dmitriy~S. Malyshev.
\newblock Vertex coloring of graphs with few obstructions.
\newblock {\em Discrete Applied Mathematics}, 216:273--280, 2017.

\bibitem{DPS19}
Barnaby Martin, Dani{\"{e}}l Paulusma, and Siani Smith.
\newblock Colouring ${H}$-free graphs of bounded diameter.
\newblock {\em Proc. MFCS 2019, LIPIcs}, 138:14:1--14:14, 2019.

\bibitem{MPS21}
Barnaby Martin, Dani{\"{e}}l Paulusma, and Siani Smith.
\newblock Colouring graphs of bounded diameter in the absence of small cycles.
\newblock {\em Proc. CIAC 2021, LNCS}, 12701:367--380, 2021.

\bibitem{MPS}
Barnaby Martin, Dani{\"{e}}l Paulusma, and Siani Smith.
\newblock Hard problems that quickly become very easy.
\newblock {\em Information Processing Letters}, 174:106213, 2022.

\bibitem{MS16}
George~B. Mertzios and Paul~G. Spirakis.
\newblock Algorithms and almost tight results for 3-colorability of small
  diameter graphs.
\newblock {\em Algorithmica}, 74:385--414, 2016.

\bibitem{MR14}
Michael Molloy and Bruce~A. Reed.
\newblock Colouring graphs when the number of colours is almost the maximum
  degree.
\newblock {\em Journal of Combinatorial Theory, Series {B}}, 109:134--195,
  2014.

\bibitem{Pa15}
Dani\"el Paulusma.
\newblock Open problems on graph coloring for special graph classes.
\newblock {\em Proc. WG 2015, LNCS}, 9224:16--30, 2015.

\bibitem{PPR21}
Marcin Pilipczuk, Michal Pilipczuk, and Pawe{\l} Rz{\k{a}}\.{z}ewski.
\newblock Quasi-polynomial-time algorithm for independent set in ${P}_t$-free
  graphs via shrinking the space of induced paths.
\newblock {\em Proc SOSA 2021}, pages 204--209, 2021.

\bibitem{Ra30}
Frank~P. Ramsey.
\newblock On a problem of formal logic.
\newblock {\em Proceedings of the London Mathematical Society}, s2-30:264--286,
  1930.

\bibitem{RS04b}
Bert Randerath and Ingo Schiermeyer.
\newblock Vertex colouring and forbidden subgraphs -- a survey.
\newblock {\em Graphs and Combinatorics}, 20:1--40, 2004.

\bibitem{RS}
Alberto Rojas and Maya Stein.
\newblock 3-colouring ${P}_t$-free graphs without short odd cycles.
\newblock {\em CoRR}, abs/2008.04845, 2020.

\bibitem{Ry94}
Zdenek Ryj{\'{a}}cek.
\newblock Almost claw-free graphs.
\newblock {\em Journal of Graph Theory}, 18:469--477, 1994.

\bibitem{Sc78}
Thomas~J. Schaefer.
\newblock The complexity of satisfiability problems.
\newblock {\em Proc. STOC 1978}, pages 216--226, 1978.

\bibitem{S68}
R.~R. Singleton.
\newblock There is no irregular {M}oore graph.
\newblock {\em Amererican Mathematical Monthly}, 75:42--43, 1968.

\bibitem{Tu97}
Zsolt Tuza.
\newblock Graph colorings with local constraints - a survey.
\newblock {\em Discussiones Mathematicae Graph Theory}, 17:161--228, 1997.

\end{thebibliography}
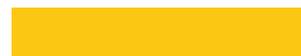
\end{document}